\documentclass[12pt]{amsart}
  \usepackage{latexsym}
  \usepackage[all]{xy}
  \usepackage{amsfonts}
  \usepackage{amsthm}
  \usepackage{amsmath}
  \usepackage{amssymb}
 \numberwithin{equation}{section}

  \def\sw#1{{\sb{(#1)}}}

  \def\<{{\langle}}
  \def\>{{\rangle}}

  \def\eps{\varepsilon}

  \def\note#1{{}}

  \def\note#1{}
  \def\M{{\mathfrak M}}
  \def\I{\mathfrak{I}}

   \def\cX{{\mathfrak X}}

  \def\bX{{\mathbb X}}
  \def\oX{\mathfrak{\bar{X}}}

  \def\lrhom#1#2#3#4{{{\rm Hom}\sb{#1, #2}(#3,#4)}}
   
  \def\lhom#1#2#3{{{\rm Hom}\sb{#1-}(#2,#3)}}
  \def\rhom#1#2#3{{{\rm Hom}\sb{#1}(#2,#3)}}

  \def\rend#1#2{{{\rm End}\sb{#1}(#2)}}

  \def\Rend#1#2{{{\rm End}\sp{#1}(#2)}}
  
  \def\Rhom#1#2#3{{{\rm Hom}\sp{#1}(#2,#3)}}

  \def\beq{\begin{equation}}
  \def\eeq{\end{equation}}
  \def\DC{{\Delta_\cC}}
  \def \eC{{\eps_\cC}}

  \def\ot{{\otimes}}

  \def\roM{\varrho^{M}}

  \newcommand{\Ra}{\Rightarrow}

  \def\coker{\mathrm{coker}}

  \newcommand{\cC}{\mathcal{C}}
  \def\ot{{\otimes}}

  \def\cp{\raise1pt\hbox{${\scriptstyle{\#}}$}}
\def\stac#1{\raise-.2cm\hbox{$\stackrel{\displaystyle\otimes}{\scriptscriptstyle{#1}}$}}

\def\cten#1{\raise-.2cm\hbox{$\stackrel{\displaystyle\widehat{\otimes}}
{\scriptscriptstyle{#1}}$}}

  \newcounter{zlist}
  \newenvironment{zlist}{\begin{list}{(\arabic{zlist})}{
  \usecounter{zlist}\leftmargin2.5em\labelwidth2em\labelsep0.5em
  \topsep0.6ex%\itemsep0.3ex plus0.2ex minus0.3ex
  \parsep0.3ex plus0.2ex minus0.1ex}}{\end{list}}

  \newcounter{blist}
  \newenvironment{blist}{\begin{list}{(\alph{blist})}{
  \usecounter{blist}\leftmargin2.5em\labelwidth2em\labelsep0.5em
  \topsep0.6ex %\itemsep0.3ex plus0.2ex minus0.3ex
  \parsep0.3ex plus0.2ex minus0.1ex}}{\end{list}}

\def\stac#1{\raise-.2cm\hbox{$\stackrel{\displaystyle\otimes}{\scriptscriptstyle{#1}}$}}

  \newtheorem{proposition}{Proposition}[section]

  \newtheorem{corollary}[proposition]{Corollary}
  \newtheorem{theorem}[proposition]{Theorem}

  \theoremstyle{definition}
  \newtheorem{definition}[proposition]{Definition}
  \newtheorem{example}[proposition]{Example}

  \theoremstyle{remark}
  \newtheorem{remark}[proposition]{Remark}

  \newcounter{c}
  \renewcommand{\[}{\setcounter{c}{1}$$}
  \newcommand{\etyk}[1]{\vspace{-7.4mm}$$\begin{equation}\Label{#1}
  \addtocounter{c}{1}}
  \renewcommand{\]}{\ifnum \value{c}=1 $$\else \end{equation}\fi}
\begin{document}

 \title{Notes on formal smoothness}
   \author{Tomasz Brzezi\'nski}
  \address{ Department of Mathematics, Swansea University,
  Singleton Park, \newline\indent  Swansea SA2 8PP, U.K.}
 \email{T.Brzezinski@swansea.ac.uk}
 \dedicatory{Dedicated to Robert Wisbauer on the occasion of his 65th birthday}

  \begin{abstract}
  The definition of an S-category is proposed by weakening the axioms of a Q-category introduced by Kontsevich and Rosenberg. Examples of Q- and S-categories and (co)smooth objects in such categories are given.
    \end{abstract}
  \maketitle

\section{Introduction}
In \cite{KonRos:non} Kontsevich and Rosenberg introduced the notion of a Q-category as a framework for developing non-commutative algebraic geometry. Relative to such a Q-category they introduced and studied the notion of a {\em formally smooth object}. Depending on the choice of Q-category this notion captures e.g.\ that of a {\em smooth algebra} of \cite{Sch:smo}, which arose  a considerable interest since its role in non-commutative geometry  was revealed in   
\cite{CunQui:alg}. 

The aim of these notes is to give a number of examples of Q-categories, and their weaker version which we term S-categories, of interest in module, coring and comodule theories, and to give examples of smooth objects in these Q-categories.  
Crucial to the definition of an S-category is the notion of a {\em separable functor} introduced in \cite{NasBer:sep}. In these notes we consider only the separability of functors with adjoints. This case is fully described by the Rafael Theorem \cite{Raf:sep}: A functor which has a right (resp.\ left) adjoint is separable if and only if the unit (resp.\ counit) of adjunction is a natural section (resp.\ retraction). For a detailed discussion of separable functors we refer to \cite{CaeMil:gen}.

Throughout these notes, by a category we mean a set-category (i.e.\ in which morphisms form sets), by functors we mean covariant functors. All rings are unital and associative. For an $A$-coring $\cC$, $\DC$ denotes the coproduct and $\eC$ denotes the counit. Whenever needed, we use the standard Sweedler notation for a coproduct $\DC(c) = \sum c\sw 1\ot_A c\sw 2$ and for a coaction $\roM(m) = \sum m\sw 0\ot_A m\sw 1$.

\section{Smoothness and cosmoothness in Q- and S-categories}\label{s-cat}
Here we gather definitions of categories and objects we study in these notes.
\begin{definition}\label{def.s-cat}
An {\em S-category} is a pair of functors ${\bX} = ( \xymatrix{{\oX} \ar@<2pt>[r]^{u_*} & \cX\ar@<2pt>[l]^{u^*}})$ such that $u^*$ is separable and left adjoint of $u_*$.
\end{definition}

This means that in an S-category $\bX = ( \xymatrix{\oX \ar@<2pt>[r]^{u_*} & \cX\ar@<2pt>[l]^{u^*}})$ the unit of adjunction $\eta: \cX\to u_*u^*$ has a natural retraction $\nu : u_*u^*\to \cX$. Therefore, for all objects $x$ of $\cX$ and $y$ of $\oX$, there exist morphisms
$$
\oX( y, u^*(x)) \to \cX(u_*(y), x), \qquad g\mapsto \nu_x\circ u_*(g).
$$
The notion of an S-category is a straightforward generalisation of that of a Q-category, introduced in \cite{KonRos:non}. The latter is defined as a pair of functors $\bX = ( \xymatrix{\oX \ar@<2pt>[r]^{u_*} & \cX\ar@<2pt>[l]^{u^*}})$ such that $u^*$ is full and faithful and left adjoint of $u_*$. In a Q-category the unit of adjunction $\eta$ is a natural isomorphism, hence, in particular, a section. Thus any Q-category is also an S-category. Following the Kontsevich-Rosenberg terminology (prompted by algebraic geometry) the functors $u_*$ and $u^*$ constituting an S-category are termed the {\em direct image} and {\em inverse image} functors, respectively.

\begin{definition}\label{def.s-cat.comp}
We say that an S-category $\bX = ( \xymatrix{\oX \ar@<2pt>[r]^{u_*} & \cX\ar@<2pt>[l]^{u^*}})$ is {\em supplemented} if there exists a functor $u_!: \oX\to \cX$ and a natural transformation $\bar{\eta}: \oX \to u^*u_!$.  
\end{definition}

In particular, an S-category $\bX = ( \xymatrix{\oX \ar@<2pt>[r]^{u_*} & \cX\ar@<2pt>[l]^{u^*}})$ is supplemented  if $u^*$ has a left adjoint. Furthermore, $\bX$ is supplemented  if the functor $u_*$ is separable, since, in this case, the counit of adjunction has a section which we can take for $\bar{\eta}$ (and $u_! = u_*$). This supplemented S-category is termed a {\em self-dual supplemented S-category}.

In a supplemented S-category, for any $y\in \oX$, there is a canonical morphism in $\cX$, natural in $y$,
$$
r_y : u_*(y) \to u_!(y),
$$
defined as a composition
$$
r_y : \xymatrix{u_*(y) \ar[rr]^ {u_*(\bar{\eta}_y)} && u_*u^*u_!(y) \ar[rr]^{\nu_{u_!(y)}} && u_!(y)\ .}
$$
 The existence of canonical morphisms $r_y$ allows us to make the following
\begin{definition} \label{def.smooth}
Given a supplemented S-category $\bX = ( \xymatrix{\oX \ar@<2pt>[r]^{u_*} & \cX\ar@<2pt>[l]^{u^*}})$, with the natural map $r: u_*\to u_!$,
 an object $x$ of $\cX$ is said to be:
\begin{blist}
\item {\em formally $\bX$-smooth} if, for any $y\in \oX$, the mapping
$\cX(x, r_y)$ is surjective;
\item {\em formally $\bX$-cosmooth} if, for any $y\in \oX$, the mapping
$\cX(r_y,x)$ is surjective.
\end{blist}
\end{definition}

\begin{remark}
We would like to stress that the notion of formal $\bX$-(co)smoothness is relative to the choice of the retraction of the unit of adjunction, and the choice of $u_!$ and $\bar{\eta}$, since the definition of $r$ depends on all these data. 
\end{remark}
Dually to S- and Q-categories one defines S$^\circ$-categories and Q$^\circ$-categories.

\begin{definition}\label{def.so}
An {\em S$^\circ$-category} (respectively {\em Q$^\circ$-category}) is a pair of functors $\bX = ( \xymatrix{\oX \ar@<2pt>[r]^{u_*} & \cX\ar@<2pt>[l]^{u^*}})$ such that $u^*$ is separable (resp.\ fully faithful) and right adjoint of $u_*$.
\end{definition}

Thus an adjoint pair of separable functors gives rise to a supplemented S- and S$^\circ$-category. In these notes (with a minor exception) we concentrate on S-categories.

\section{Examples of Q- and S-categories}
The following generic example of a Q-category was constructed by Kontsevich and Rosenberg in \cite{KonRos:non}.
\begin{example}[The Q-category of morphisms] \label{ex.mor} Let $\cX$ be any category, and let $\cX^2$ be the category of morphisms in $\cX$ defined as follows. The objects of $\cX^2$ are morphisms  $f$, $g$ in $\cX$. Morphisms in $\cX^2$ are commutative
squares
\[
{\xymatrix{ x\ar[r]^f \ar[d] & y\ar[d] \\x' \ar[r]^g & y'}}
\]
where the vertical arrows are  in $\cX$. Now, set $\oX = \cX^2$. The inverse image
functor $u^{\ast}$ is
\[
u^{\ast}:x\mapsto \left(\xymatrix{x \ar[r]^x& x}\right) ,\qquad \left(
{\xymatrix{x\ar[r]^f & y}} \right)  \mapsto\left(\vcenter{{\xymatrix{ x\ar[r]^x \ar[d]_{f} & x\ar[d]^{f} \\
y \ar[r]^y & y}}}
\right)  .
\]
The direct image functor $u_{\ast}$ is defined by
\[
u_{\ast}:\left(
{\xymatrix{x\ar[r]^f & y}} 
\right)  \mapsto x,\qquad
\vcenter{{\xymatrix{ x\ar[r]^f \ar[d] & y\ar[d] \\x' \ar[r]^g & y'}}} \mapsto\left(\vcenter{{\xymatrix{x\ar[d]\\ x'}}}  \right)  .
\]
Note that, for all objects $x$ and morphisms $f$ in $\cX$, 
$$
u_*u^*(x) = u_*(\xymatrix{x \ar[r]^x &x}) = x, \qquad u_*u^*(f) = f.
$$
Hence, for all objects $x$ in $\cX$, there is an isomorphism (natural in $x$), $\eta_x: x\to u_*u^*(x)$, $\eta_x =x$.
 
Note further that for all objects $\xymatrix{x \ar[r]^f &y}$ in $\cX^2$, $u^*u_*(f) = x$, and we can define a morphism $\eps_f : u^*u_*(f) \to f$ by
$$
\eps_f = \left(\vcenter{\xymatrix{ x\ar[r]^x \ar[d]_{x} & x\ar[d]^{f} \\
x \ar[r]^f & y}}\right).
$$
In this way, $u_*$ is the right adjoint of $u^*$ with counit $\eps$ and unit $\eta$. The unit is obviously a natural isomorphism, hence $u^*$ is full and faithful and, thus, a Q-category $\bX = ( \xymatrix{\oX \ar@<2pt>[r]^{u_*} & \cX\ar@<2pt>[l]^{u^*}})$ is constructed. $\bX$ is supplemented, since $u^*$ has a left adjoint 
$$
u_! : \left( \xymatrix{x\ar[r]^f & y}\right) \mapsto y, \qquad \left(\vcenter{\xymatrix{ x\ar[r]^f \ar[d] & y\ar[d] \\x' \ar[r]^g & y'}}\right) \mapsto \left(\vcenter{\xymatrix{  y\ar[d] \\ y'}}\right). 
$$
The unit of the adjunction $u_!\dashv u^*$  is, for all $f:x\to y$,
$$
\bar{\eta}_f = \left(\vcenter{\xymatrix{ x\ar[r]^f \ar[d]_{f} & y\ar[d]^{y} \\
y \ar[r]^y & y}}\right),
$$
and thus the corresponding maps $r$ come out as
$$
r_f =f.
$$
Consequently, an object $x\in \cX$ is formally $\bX$-smooth (when $\bX$ is supplemented by $u_!$ and $\bar{\eta}$) provided, for all $\xymatrix{y \ar[r]^f &z} \in \oX$, the mapping
$$
\cX(x,y)\to \cX(x,z), \qquad g\mapsto  f\circ g,
$$
is surjective. Similarly, $x$ is formally $\bX$-cosmooth if and only if the mappings
$$
\cX(z,x)\to \cX(y,x), \qquad g\mapsto g\circ f,
$$ 
are surjective.

This generic example has a useful modification whereby one takes for $\oX$ any full subcategory of $\cX^2$ which contains all the identity morphisms in $\cX$.
\end{example}
\begin{example}[The Wisbauer Q-category] \label{ex.sigma} Let $R$ by a ring and $M$ be a left $R$-module. Following \cite[Section~15]{Wis:fou} $\sigma[M]$ denotes a full subcategory of  the category ${}_R\M$ of left $R$-modules, consisting of objects subgenerated by $M$. Since $\sigma[M]$ is a full subcategory of ${}_R\M$, the inclusion functor
$$
u^*: \sigma[M] \to {}_R \M,
$$
is full and faithful. It also has the right adjoint, the trace functor (see \cite[45.11]{Wis:fou} or \cite[41.1]{BrzWis:cor}),
$$
u_*= \mathcal{T}^M: {}_R \M \to \sigma[M], \qquad  \mathcal{T}^M(L) = \sum \{f(N)\; |\; N\in\sigma[M], \; f\in \rhom R ML\}.
$$
Hence there is a Q-category $\bX = ( \xymatrix{\oX \ar@<2pt>[r]^{u_*} & \cX\ar@<2pt>[l]^{u^*}})$ with $\cX=\sigma[M]$ and $\oX = {}_R\M$.
\end{example}
All the remaining examples come from the theory of corings.

\begin{example}[Comodules of a locally projective coring]\label{ex.loc.proj} This is a special case of Example~\ref{ex.sigma}. Let $(\cC,\DC,\eC)$ be an $A$-coring which is locally projective as a left $A$-module. Let $R ={}^*\cC = \lhom A\cC A$ be a left dual ring of $\cC$ with the unit $\eC$ and product, for all $r,s\in R$,
$$
rs : \xymatrix{\cC \ar[r]^\DC & \cC\ot_A\cC \ar[rr]^{\cC\ot_A s} && \cC \ar[r]^r & A.}
$$
Take $\cX = \M^\cC$, the category of right $\cC$-comodules, and $\oX = {}_R\M$. Define a functor
$$
u^* : \M^\cC \to {}_R\M, \qquad M\mapsto M,
$$
where right $\cC$-comodule $M$ is given a left $R$-module structure by $rm = \sum m\sw 0r(m\sw 1)$. Since $\cC$ is a locally projective left $A$-module, the functor $u^*$ has a right adjoint, the rational functor (see \cite[20.1]{BrzWis:cor}),
$$u_* = \mathrm{Rat}^\cC: {}_R\M\to \M^\cC , \qquad \mathrm{Rat}^\cC(M) = \{ n\in M \; |\; \mbox{ $n$ is rational}\},
$$
where an element $n\in M$ is said to be rational provided there exists $\sum_i m_i\ot_A c_i \in M\ot_{A}\cC$ such that, for all $r\in R$, $rm = \sum_i m_i r(c_i)$. Here, the left $R$-module $M$ is seen as a right $A$-module via the anti-algebra map $A\to R$, $a\mapsto \eC(-a)$.
\end{example}

\begin{example}[Coseparable corings] \label{ex.cosep} Recall that an $A$-coring $(\cC,\DC,\eC)$ is said to be {\em coseparable} \cite{Guz:coi} if there exists a $(\cC,\cC)$-bicomodule retraction of the coproduct $\DC$. This is equivalent to the existence of a {\em cointegral} defined as an $(A,A)$-bimodule map $\delta: \cC\ot_A\cC\to A$ such that $\delta\circ\DC = \eC$, and 
$$
(\cC\ot_A\delta)\circ(\DC\ot_A\cC) = (\delta\ot_A \cC)\circ (\cC\ot_A\DC).
$$
Furthermore, this is equivalent to the separability of the forgetful functor $(-)_A:\M^\cC\to \M_A$ \cite[Theorem~3.5]{Brz:str}). Since this forgetful functor is a left adjoint to $-\ot_A\cC: \M_A\to \M^\cC$, a coseparable coring $\cC$ gives rise to an S-category $\bX$ with 
$$
\cX = \M^\cC, \qquad \oX = \M_A, \qquad u^* = (-)_A, \qquad u_* = -\ot_A\cC.
$$
This S-category is denoted by $\bX^\cC_\delta$. By \cite[Theorem~3.5]{Brz:str}, the retraction $\nu$ of the unit of the adjunction is given explicitly, for all $M\in \M^\cC$,
$$
\nu_M : M\ot_A\cC\to M, \qquad m\ot_A c\mapsto \sum m\sw 0\delta(m\sw 1\ot_A c).
$$
In general, $\bX^\cC_\delta$ need not to be supplemented. However, if there exists
$$
e\in \cC^A : = \{c\in \cC\: |\; \forall a\in A, \; ac=ca\},
$$
then $\bX^\cC_\delta$ can be supplemented with
$$
u_! = -\ot_A\cC, \qquad \bar{\eta}_M : M\to M\ot_A\cC, \qquad m\mapsto m\ot_A e.
$$
This supplemented S-category is denoted by $\bX^\cC_{\delta,e}$.

Recall that an $A$-coring $\cC$ is said to be {\em cosplit} if there exists an $A$-central element $e\in \cC^A$ such that $\eC(e) =1$. By \cite[Theorem~3.3]{Brz:str} this is equivalent to the separability of the functor $-\ot_A\cC$, and thus a cosplit coring gives rise to an S$^\circ$-category. Therefore, a coring which is both cosplit and coseparable  induces a self-dual, supplemented $S$-category.
\end{example}

In addition to the defining adjunction of an $A$-coring, $(-)_A\dashv -\otimes_A\cC$, for any right $\cC$-comodule $P$, there is a pair of adjoint functors
$$
-\ot_B P : \M_B\to \M^\cC, \qquad \Rhom\cC P-:\M^\cC\to \M_B,
$$
where $B$ is any subring of the endomorphism ring $S=\Rend\cC P$ (cf.\ \cite[18.21]{BrzWis:cor}. Depending on the choice of $\cC$, $P$ and $B$ this adjunction provides a number of examples of Q-categories.
\begin{example}[Comatrix corings]\label{ex.comat} Take a $(B,A)$-bimodule $P$ that is finitely generated and projective as a right $A$-module. Let $\mathbf{e} \in P\otimes_A P^*$ be the dual basis (where $P^* = \rhom APA$), and let $\cC = P^*\otimes_B P$ be the comatrix coring associated to $P$ \cite{ElKGom:com}. The coproduct and counit in $\cC$ are given by 
$$
\DC (\xi\ot_Bp) = \xi\ot_B \mathbf{e}\ot_B p, \qquad \eC(\xi\ot_B p) = \xi(p),
$$
for all $p\in P$ and $\xi\in P^*$. $P$ is a right $\cC$-comodule with the coaction $\varrho^P: p\mapsto \mathbf{e}\ot_Bp$. Let
$$
\cX = \M_B, \qquad \oX= \M^\cC, \qquad u^* = -\ot_BP, \qquad u_*= \Rhom \cC P -.
$$
In view of \cite[Proposition~2.3]{CaeDeG:com}, $\bX = ( \xymatrix{\oX \ar@<2pt>[r]^{u_*} & \cX\ar@<2pt>[l]^{u^*}})$ is a Q-category if and only if the map
$$
B\to P\ot_A P^*, \qquad b\mapsto b\mathbf{e},
$$
is pure as a morphism of left $B$-modules (equivalently, $P$ is a totally faithful left $B$-module).
\end{example}
\begin{example}[Strongly $(\cC,A)$-injective comodules] \label{ex.str.inj}
Let $\cC$ be an $A$-coring, let $P$ be a right $\cC$-comodule and $S=\Rend\cC P$.
Following \cite[2.9]{Wis:gal}, $P$ is said to be {\em strongly $(\cC,A)$-injective} if the coaction $\varrho^P: P\to P\ot_A\cC$ has a left $S$-module right $\cC$-comodule retraction. For such a comodule, define
$$
\cX = \M_S, \qquad \oX= \M^\cC, \qquad u^* = -\ot_SP, \qquad u_*= \Rhom \cC P -.
$$
In view of \cite[3.2]{Wis:gal}, if $P$ is a finitely generated and projective as a right $A$-module, then $\bX = ( \xymatrix{\oX \ar@<2pt>[r]^{u_*} & \cX\ar@<2pt>[l]^{u^*}})$ is a Q-category.
\end{example}
\begin{example}[$(\cC,A)$-injective Galois comodules]\label{ex.gal} Recall that a right $\cC$-comodule is said to be {\em $(\cC,A)$-injective}, provided there is a right $\cC$-colinear retraction of the coaction. The full subcategory of $\M^\cC$ consisting of all $(\cC,A)$-injective comodules is denoted by $\I^\cC$.

Let $P$ be a right comodule of an $A$-coring $\cC$, and let $S = \Rend\cC P$ and $T=\rend AP$. Following \cite[4.1]{Wis:gal}, $P$ is said to be a {\em Galois comodule} if, for all $N\in \I^\cC$, the evaluation map
$$
\Rhom\cC P N \ot_S P\to N, \qquad f\ot_Sp\to f(p),
$$
is an isomorphism of right $\cC$-comodules.

Let $P$ be a Galois comodule, and assume that the inclusion $S\to T$ has a right $S$-module retraction. By \cite[4.3]{Wis:gal} this is equivalent to say that $P$ is a $(\cC,A)$-injective comodule, and hence one can consider the following pair of categories and adjoint functors:
$$
\oX = \M_S, \qquad \cX=\I^\cC, \qquad u_* = -\otimes_SP: \oX\to \cX, \qquad u^* =\Rhom \cC P-: \cX\to\oX.
$$
Since the evalutaion map is the counit of the adjunction $u_* \dashv u^*$, the Galois property of $P$ means that the functor $u^*$ is fully faithful. Thus $\bX = ( \xymatrix{\oX \ar@<2pt>[r]^{u_*} & \cX\ar@<2pt>[l]^{u^*}})$ is a Q$^\circ$-category.
\end{example}
\section{Examples of smooth and cosmooth objects}

Let $\cC$ be an $A$-coring, set $\cX=\M^\cC$, and consider the full subcategory of $\cX^2$ consisting of all monomorphisms in $\M^\cC$ with an $A$-module retraction. With these data one constructs a Q-category as in Example~\ref{ex.mor}. This Q-category is denoted by $\bX^\cC$.

\begin{theorem}\label{thm.inj}
A right $\cC$-comodule $M$ is $(\cC,A)$-injective if and only if $M$ is a formally $\bX^\cC$-cosmooth object.
\end{theorem}
\begin{proof} In view of the discussion at the end of Example~\ref{ex.mor}, an object $M\in \cX=\M^\cC$ if formally $\bX^\cC$-cosmooth if and only if, for all morphisms $f: N\to N'$ in $\M^\cC$ with right $A$-module retraction, the maps
$$
\vartheta_f: \Rhom \cC {N'} M \to \Rhom \cC N M, \quad g\mapsto g\circ f,
$$
are surjective. This means that, for all $h\in \Rhom \cC N M$, there is $g\in \Rhom \cC {N'} M$ completing the following diagram
$$
\xymatrix{ && M &\\
0 \ar[r] & N\ar[ru]^h\ar@<2pt>[rr]^f && N'\ar@<2pt>[ll]\ar@{.>}[lu]_g,}
$$
where the arrow $N'\to N$ is in $\M_A$, and thus is equivalent to $M$ being $(\cC,A)$-injective, see \cite[18.18]{BrzWis:cor}.
\end{proof}

The arguments used in the proof of Theorem~\ref{thm.inj}, in particular, the identification of (co)smooth objects as object with a (co)splitting property, apply to all Q-categories of the type described in Example~\ref{ex.mor}. This leads to reinterpretation of smooth algebras and coalgebras in abelian monoidal categories studied in \cite{ArdMen:Hoc}. 

\begin{example}\label{ex.smoalg}
Let $(V, \otimes)$ be an abelian monoidal category, i.e.\ a monoidal category which is abelian and such that the tensor functors $-\otimes v$, $v\otimes -$ are additive and right exact, for all objects $v$ of $V$. Let $\cX$ be the category of algebras in $V$, and let $\oX$ be a full subcategory of $\cX^2$, consisting of {\em Hochschild algebra extensions}, i.e.\ of all surjective algebra morphisms split as morphisms in $V$ and with a square-zero kernel. Denote the resulting Q-category by $\mathbb{HAE}$. In view of \cite[Theorem~3.8]{ArdMen:Hoc}, an algebra in $V$ is formally smooth in the sense of \cite[Definition~3.9]{ArdMen:Hoc}, i.e.\ it has the Hochschild dimension at most 1, if and only if it is a formally $\mathbb{HAE}$-smooth object.

In particular if $(V, \otimes)$ is the category of vector spaces (with the usual tensor product), we obtain the characterisation of smooth algebras \cite{Sch:smo} (or semi-free algebras in the sense of \cite{CunQui:alg}), described in \cite[Proposition~4.3]{KonRos:non}.
\end{example}
\begin{example}\label{ex.smocoalg}
Let $(V, \otimes)$ be an abelian monoidal category. Let $\cX$ be the category of coalgebras in $V$, and let $\oX$ be a full subcategory of $\cX^2$, consisting of {\em Hochschild coalgebra extensions}, i.e.\ of all injective coalgebra morphisms $\sigma: C\to E$ split as morphisms in $V$ and with the property $(p\ot p)\circ \Delta_E =0$, where $p: E\to \coker \sigma$ is the cokernel of $\sigma$. Denote the resulting Q-category by $\mathbb{HCE}$. In view of \cite[Theorem~4.16]{ArdMen:Hoc}, a coalgebra in $V$ is formally smooth in the  sense of \cite[Definition~4.17]{ArdMen:Hoc} if and only if it is a formally $\mathbb{HCE}$-cosmooth object.
\end{example}

The following example is taken from \cite{ArdBrz:for}.
\begin{example}\label{ex.smodule}
Let $A$ and $B$ be rings, and let $M$ be  a $(B,A)$-bimodule.  Denote by $\mathcal{E}_M$ the class of all $(B,B)$-bilinear maps $f$ such that $\rhom B M f$ splits as an $(A,B)$-bimodule map.  A $B$-bimodule $P$ is said to be {\em $\mathcal{E}_M$-projective}, provided every morphism $N\to P$ in $\mathcal{E}_M$ has a section. By the argument dual to that in the proof of Theorem~\ref{thm.inj} one can reinterpret $\mathcal{E}_M$-projectivity as formal smoothness as follows.

Take $\cX$ to be the category of $B$-bimodules and $\oX = \mathcal{E}_M$, a full subcategory of $\cX^2$. Denote the resulting Q-category by $\mathbb{E}$. A $B$-bimodule $P$ is formally $\mathbb{E}$-smooth if and only if, for all 
$f: N\to N' \in \mathcal{E}_M$, the function
$$
\Theta(f): \lrhom BB{P} N \to \lrhom BB{P}{N'}, \qquad g\mapsto f\circ g,
$$
is surjective. In terminology of \cite[Chapter X]{HilSta:cou}, $\mathbb{E}$-smoothness of $P$ is equivalent to 
the $\mathcal{E}_{M}$-projectivity of $P$. 

A $(B,A)$-bimodule $M$  is said to be {\em formally smooth} provided  the kernel of the evaluation map  
\[
\mathrm{ev}_{M}:M\otimes_{A}{\rhom BMB}\rightarrow B,\qquad\mathrm{ev} 
_{M}\left(  m\otimes_{A}f\right)  = f(m).
\]
is an
$\mathcal{E}_{M}$-projective $B$-bimodule. Thus $M$ is formally smooth if and only if $\ker \mathrm{ev}_{M}$ is formally $\mathbb{E}$-smooth.
\end{example}

Next we characterise all  smooth and cosmooth objects in the supplemented S-category $\bX^\cC_{\delta, e}$ associated to a coseparable $A$-coring $\cC$ with an $A$-central element $e$ as in Example~\ref{ex.cosep}.

\begin{proposition}\label{prop.cosep.smooth}
Let $\cC$ be a coseparable $A$-coring with a cointegral $\delta$ and an $A$-central element $e$, and let $\bX^\cC_{\delta, e}$ be the associated supplemented S-category. A right $\cC$-comodule $M$ is formally $\bX^\cC_{\delta, e}$-smooth if and only if the map
$$
\kappa_M :M\to M, \qquad m\mapsto \sum m\sw 0\delta(e\ot_A m\sw 1),
$$
is a right $A$-linear section (i.e.\ $\kappa_M$ has a left inverse in $\rend AM$).
\end{proposition}
\begin{proof}
In this case, for all $N\in \M_A$, the canonical morphisms $r_N$ read
$$
r_N : N\ot_A\cC\to N\ot_A\cC, \qquad n\ot_A c\mapsto \sum n\ot_A e\sw 1\delta(e\sw 2\ot_A c).
$$
Using the (defining adjunction) isomorphisms $\Rhom\cC M{N\ot_A\cC}\simeq \rhom AMN$, the maps
$$
\Rhom \cC M {r_N} : \Rhom\cC M{N\ot_A\cC}\to \Rhom\cC M{N\ot_A\cC},
$$
can be identified with
$$
\vartheta_{M,N} : \rhom AMN\to \rhom AMN, \qquad f\mapsto (N\ot_A\eC)\circ r_N\circ (f\ot_A\cC)\circ \roM,
$$
where $\roM: M\to M\ot_A\cC$ is the coaction. Hence $\Rhom \cC M {r_N}$ are surjective for all $N$ if and only if $\vartheta_{M,N}$ are surjective for all $N$. These can be computed further, for all $m\in M$, $f\in \rhom AMN$,
\begin{eqnarray*}
\vartheta_{M,N}(f)(m) &=&  (N\ot_A\eC)\circ r_N(f(m\sw 0)\ot_A m\sw 1)\\
&=& (N\ot_A\eC)(f(m\sw 0)\ot_A e\sw 1\delta(e\sw 2\ot_Am\sw 1)) = 
\sum f(m\sw 0)\delta(e\ot_A m\sw 1) \\
&=& \sum f(m\sw 0\delta(e\ot_A m\sw 1)) = f(\kappa_M(m)),
\end{eqnarray*}
by the right $A$-linearity of $f$. Hence
$$
\vartheta_{M,N} (f)  = f\circ\kappa_M.
$$
If $\kappa_M$ has a retraction $\lambda_M\in \rend AM$, then for all $f\in \rhom AMN$,
$$
\vartheta_{M,N} (f\circ\lambda_M) = f\circ \lambda_M\circ\kappa_M = f,
$$
i.e., the $\vartheta_{M,N}$ are surjective. If, on the other hand, all the $\vartheta_{M,N}$ are surjective, choose $N=M$ and take any $\lambda_M\in \vartheta^{-1}_{M,M}(M)$. Then
$$
M = \vartheta_{M,M}(\lambda_M) = \lambda_M\circ\kappa_M,
$$
so $\lambda_M$ is a retraction of $\kappa_M$ as required.
\end{proof}

\begin{example}[Modules graded by $G$-sets]\label{gr.mod} 
Let $G$ be a group, $X$ be a (right) $G$-set and let $A = \oplus_{\sigma \in G}$ be a $G$-graded $k$-algebra. Following \cite{NasRai:mod}, a $kX$-graded right $A$-module $M = \oplus_{x\in X} M_x$ is said to be {\em graded by $G$-set $X$} provided, for all $x\in X$, $\sigma\in G$,
$$
M_xA_\sigma \subseteq M_{x\sigma}.
$$
A morphism of such modules is an $A$-linear map which preserves the $X$-grading. The resulting category is denoted by $\mbox{gr-}(G,A,X)$. It is shown in \cite[Section~4.6]{CaeMil:gen} that $\mbox{gr-}(G,A,X)$ is isomorphic to the category of right comodules of the following coring $\cC$. As a left $A$-module $\cC = A\ot kX$. The right $A$-multiplication is given by
$$
(a\ot x)a_\sigma = aa_\sigma \ot x\sigma, \qquad \forall a\in A, x\in X, a_\sigma\in A_\sigma.
$$
The coproduct and  counit are defined by
$$
\DC(a\ot x) = (a\ot x)\ot_A (1_A\ot x), \qquad \eC(a\ot x) = a.
$$
An object $M= \oplus_{x\in X} M_x$ in $\mbox{gr-}(G,A,X)$ is a right $\cC$-comodule with the coaction $\roM:M\to M\ot_A\cC$, $m_x\mapsto m_x\ot_A 1_A\ot x$, where $m_x\in M_x$. Also in \cite[Section~4.6]{CaeMil:gen} it is shown that $\cC$ is a coseparable coring with a cointegral (cf.\ \cite[Proposition~2.5.3]{Zar:adj})
$$
\delta: \cC\ot_A\cC\simeq A\ot kX\ot kX \to A, \qquad a\ot x\ot y\mapsto a \delta_{x,y}.
$$
Thus $\mbox{gr-}(G,A,X)$ gives rise to an S-category as in Example~\ref{ex.cosep}. 

Let $X^G := \{ x\in X\; |\; \forall \sigma\in G, \; x\sigma =x\}$ be the set of one-point orbits of $G$ in $X$. If $X^G\neq\emptyset$,  the above S-category can be supplemented as in Example~\ref{ex.cosep} by
$$
e := 1_A\ot z, \qquad z\in X^G.
$$
In this case, for any $M\in \mbox{gr-}(G,A,X)$, the map $\kappa_M$ in Proposition~\ref{prop.cosep.smooth} comes out as
$$
\kappa_M(m_x) = m_x\delta_{x,z}, \qquad \forall m_x\in M_x.
$$
Thus a graded module $M\in \mbox{gr-}(G,A,X)$ is formally $\bX^{A\ot kX}_{\delta, e}$-smooth if and only if it is concentrated in degree $z$, i.e., $M=M_z$.
\end{example}

Given an $A$-coring $\cC$, the set of right $A$-module maps $\cC\to A$, $\cC^*$, is a ring with the unit $\eC$ and the product, for all $\xi,\xi'\in \cC^*$, 
$$
\xi\xi' : \xymatrix{\cC \ar[r]^\DC & \cC\ot_A\cC \ar[rr]^{\xi'\ot_A \cC} && \cC \ar[r]^\xi & A.}
$$
\begin{proposition}\label{prop.cosep.cosmooth}
Let $\cC$ be a coseparable $A$-coring with a cointegral $\delta$ and an $A$-central element $e$, and let $\bX^\cC_{\delta, e}$ be the associated supplemented S-category. Then the following statements are equivalent:
\begin{zlist}
\item All right $\cC$-comodules are formally $\bX^\cC_{\delta, e}$-cosmooth.
\item The right $A$-linear map
$$
\lambda: \cC\to A, \qquad c\mapsto \delta(e\ot_Ac),
$$
has a left inverse in the dual ring $\cC^*$.
\item The regular right $\cC$-comodule $\cC$ is formally $\bX^\cC_{\delta, e}$-cosmooth.
\end{zlist}
\end{proposition}
\begin{proof}
Note that $\cC^*$ can be identified with $\Rend \cC \cC$ via the map $\xi\mapsto (\xi\ot_A\cC)\circ\DC$ (with the inverse $f\mapsto \eC\circ f)$. Under this identification the product in $\cC^*$ coincides with the composition in $\Rend  \cC \cC$. Hence (2) is equivalent to saying that the map $r_A = (\lambda\ot_A\cC)\circ\DC$ has a retraction in $\M^\cC$. Denote this retraction by $s_A$. Note further that, since $\delta$ is a cointegral, the $r_N$ defined in the proof of Proposition~\ref{prop.cosep.smooth} can be written as $r_N =N\ot_Ar_A$. This implies that $s_A$ is a section of $r_A$ if and only if  $s_N = N\ot_A s_A$ is a retraction of $r_N = N\ot_A r_A$, for all right $A$-modules $N$. Finally observe that for all $M\in \M^\cC$ and $N\in \M_A$, the maps $\varphi_{M,N}:= \Rhom\cC {r_N} M$ come out explicitly as
$$
\varphi_{M,N} : \Rhom\cC{N\ot_A\cC}M \ni f\mapsto f\circ r_N\in  \Rhom\cC{N\ot_A\cC}M.
$$

(2) $\Ra$ (1)  The property  $s_N\circ r_N = N\ot_A\cC$, implies that, for all right $\cC$-comodules $M$ and right $A$-modules $N$, the maps $\varphi_{M,N}$ are surjective. Hence all right $\cC$-comodules are formally $\bX^\cC_{\delta, e}$-cosmooth.

The implication (1) $\Ra$ (3) is obvious.

(3) $\Ra$ (2) If $\cC$ is formally $\bX^\cC_{\delta, e}$-cosmooth, then $\varphi_{\cC,A}: \Rend \cC\cC\to \Rend \cC\cC$ is surjective. Hence there exists $s_A \in \Rend \cC\cC$ such that
$$
\cC = \varphi_{\cC, A}(s_A) = s_A\circ r_A.
$$
This completes the proof.
\end{proof}

A coseparable $A$-coring $\cC$ with a cointegral $\delta$ is said to be {\em Frobenius-coseparable} if there exists $e\in \cC^A$ such that, for all $c\in \cC$, $\delta(c\ot_A e) = \delta(e\ot_A c) = \eC(c)$. The  element $e$ is called a {\em Frobenius element}. In particular a Frobenius-coseparable coring is a Frobenius coring, see \cite[27.5]{BrzWis:cor}.

\begin{corollary}\label{cor.frob}
Let $\cC$ be a Frobenius-coseparable $A$-coring with cointegral $\delta$ and Frobenius element $e$. Then any right $\cC$-comodule is formally $\bX^\cC_{\delta, e}$-cosmooth and $\bX^\cC_{\delta, e}$-smooth.
\end{corollary}
\begin{proof}
The maps $\kappa_M$ in Proposition~\ref{prop.cosep.smooth} are all identity morphisms, hence they are sections and thus every right $\cC$-comodule is formally 
$\bX^\cC_{\delta, e}$-smooth. The map $\lambda$ in Proposition~\ref{prop.cosep.cosmooth} coincides with the counit $\eC$. Since $\eC$ is a unit in $\cC^*$, it has a left inverse, and thus every right $\cC$-comodule is formally 
$\bX^\cC_{\delta, e}$-cosmooth.
\end{proof}

\end{document}